\documentclass[12pt,twoside,final,psamsfonts]{amsart}

\usepackage[psamsfonts]{amssymb}
\usepackage{times,a4wide}
\usepackage{MnSymbol}

\usepackage{color}		% Need the color package
\usepackage{epsfig}

\usepackage{graphicx}

\newcommand{\nc}{\newcommand}

\numberwithin{equation}{section}
\theoremstyle{plain}

\newtheorem{theorem}[equation]{Theorem}

\theoremstyle{definition}

\newcommand{\C}{{\mathbb C}}
\newcommand{\N}{{\mathbb N}}
\newcommand{\D}{{\mathbb D}}
\newcommand{\R}{{\mathbb R}}

\newcommand{\T}{{\mathbb T}} %\mathbb D}}

\newcommand{\lra}{\longrightarrow}

\newcommand{\lmto}{\longmapsto}
\newcommand{\eps}{\varepsilon}
\newcommand{\vp}{\varphi}
\nc{\bea}{\begin{eqnarray}}
\nc{\eea}{\end{eqnarray}}
\nc{\beqa}{\begin{eqnarray*}}
\nc{\eeqa}{\end{eqnarray*}}
\nc{\Hi}{H^{\infty}}
\nc{\loi}{\ell^{\infty}}
\nc{\NL}{N^+\vert \Lambda}
\nc{\hf}{{\mathcal H}_{\phi}}
\nc{\liL}{\lambda\in\Lambda}
\nc{\nn}{\nonumber}
\nc{\dst}{\displaystyle}

\newenvironment{proof*}{\vskip 2mm\noindent {}}{$\blacksquare$ \vskip 2mm}
\numberwithin{equation}{section}

\renewcommand{\Re}{\mbox{Re}}
\renewcommand{\Im}{\mbox{Im}}

\title{Bad boundary behavior in star invariant subspaces II}

\author{Andreas Hartmann \& William T.\ Ross}

\address{Institut de Math\'ematiques de Bordeaux,
Universit\'e Bordeaux I, 351 cours de la Lib\'eration,
33405 Talence, France}

\address{Department of Mathematics, University of Richmond, VA 23173, USA}

\email{hartmann@math.u-bordeaux.fr, wross@richmond.edu}

\date{\today}



\begin{document}

\bibliographystyle{amsalpha}

\begin{abstract}
We continue our study begun in  \cite{HR3} concerning the radial growth of functions in the model spaces $(I H^2)^{\perp}$.
\end{abstract}

\maketitle

%\tableofcontents

\section{Introduction}

Suppose $I = B S_{\mu}$ is an inner function with Blaschke factor $B$, with zeros $\{\lambda_n\}_{n \geq 1}$ in the open unit disk $\D$ repeated according to multiplicity, and singular inner factor $S_{\mu}$ with associated positive singular measure $\mu$ on the unit circle $\T$. 
The following result was shown by Frostman in 1942 for Blaschke products
(see \cite{Frost} or \cite{CL}) and by Ahern-Clark for general
inner functions \cite[Lemma 3]{AC71}.

\begin{theorem}[Frostman, 1942; Ahern-Clark, 1971]\label{thm1.1}
Let $\zeta \in \T$ and $I$ be inner with $\mu(\{\zeta\}) = 0$. Then the following assertions are equivalent.
\begin{enumerate}
\item Every divisor of $I$ has a radial limit of modulus one at $\zeta$. 
\item Every divisor of $I$ has a radial limit at $\zeta$. 
\item The following condition holds
\bea\label{F1}
 \sum_{n \geq 1}\frac{1-|\lambda_n|}{|\zeta-\lambda_n|}+
 \int_{\T}\frac{1}{|\zeta-e^{it}|}d\mu(e^{it})<\infty.
\eea
\end{enumerate}
\end{theorem}

Based on a stronger condition than the above, Ahern and Clark \cite{AC70} were
able to characterize ``good'' non-tangential boundary behavior of functions in 
the model spaces $(I H^2)^{\perp}$ of the classical Hardy space $H^2$ (see \cite{Niktr} for a very complete treatment
of model spaces). 

\begin{theorem}[\cite{AC70}]  \label{AC-paper}
Let $I=B S_{\mu}$ be an inner function with 
zeros $\{\lambda_{n}\}_{n \geq 1}$ and associated
singular measure $\mu$. For 
$\zeta \in \T$, the following are equivalent: 
\begin{enumerate}
\item Every $f \in (I H^2)^{\perp}$ has a radial limit at $\zeta$.
\item  The following condition holds
\begin{equation} \label{AC1}
\sum_{n \geq 1} \frac{1 - |\lambda_{n}|}{|\zeta - \lambda_{n}|^2} 
 + \int_{\T}\frac{1}{|\zeta-e^{it}|^2}d\mu(e^{it})< \infty.
\end{equation}
\end{enumerate}
\end{theorem}

In this paper, we will study what happens when we are somewhere in between the Frostman condition \eqref{F1} and the Ahern-Clark condition \eqref{AC1}. In order to do
so we will introduce an auxiliary function.  
%which will allow us to measure where 
Let $\vp:(0, +\infty)\to \R^+$ be a positive increasing function such that
\begin{enumerate}
\item $x\to \frac{\dst \vp(x)}{\dst x}$ is bounded,
\item $x\lmto \frac{\dst \vp(x)}{\dst x^2}$ is decreasing,
\item $\vp(x) \asymp \vp(x+o(x))$, $x\downarrow 0$.
\end{enumerate}
Such a function $\vp$ will be called \emph{admissible}. One can check that functions like $\vp(x) = x^{p}, 1 \leq p < 2$, and $\vp(x) = x^{p} \log(1/x)$, $1 < p < 2$, are admissible. 
Our main result is the following.

\begin{theorem}\label{main-thm}
Let $I=B S_{\mu}$ be an inner function with 
zeros $\{\lambda_{n}\}_{n \geq 1}$ and associated
singular measure $\mu$,
$\vp$ an admissible function, and $\zeta\in\T$. 
If
\bea\label{grcond}
  \sum_{n \geq 1}\frac{1-|\lambda_n|}{\vp(|\zeta-\lambda_n|)}+
 \int_{\T}\frac{1}{\vp(|\zeta-e^{it}|)}d\mu(e^{it})<\infty,
\eea
then every $f \in (I H^2)^{\perp}$ satisfies 
\begin{equation} \label{main-thm-est}
 |f(r \zeta)|\lesssim  \frac{\sqrt{\vp(1-r)}}{1-r}.
\end{equation}
\end{theorem}

When $\vp(x)=x$ then we are in the Frostman situation \eqref{F1} and no
restriction is given for the growth of $f$ since generic functions in $H^2$ satisfy the growth condition 
$$|f(r \zeta)| = o(\frac{1}{\sqrt{1 - r}})$$ On the other hand, when  $\vp(x)=x^2$ we reach the
Ahern-Clark situation \eqref{AC1} . For other $\vp$ such as $\vp(x) = x^{3/2}$ or perhaps $\vp(x) = x^2 \log(1/x)$ we get that even though functions in $(I H^2)^{\perp}$ can be poorly behaved (as in the title of this paper), the growth is controlled. 

There is some history behind these types of problems. When $\vp(x) = x^{2 N + 2}$, where $N = 0, 1, 2, \cdots$,  Ahern and Clark \cite{AC70} 
showed that \eqref{grcond} is equivalent to the condition that $f^{(j)}$, $0 \leq j \leq N$, have radial limits at $\zeta$ for every $f \in (I H^2)^{\perp}$. When $\vp(x) = x^{p}$, $p \in (1, \infty)$, Cohn \cite{cohn} showed that \eqref{grcond} is equivalent to the condition that every $f \in H^{q} \cap I \overline{H_{0}^{q}}$, where $q = p (p - 1)^{-1}$, has a finite radial limit at $\zeta$.  

Why did we write this second paper? In \cite{HR3} we discussed controlled growth of functions from $(B H^2)^{\perp}$, where $B$ is a Blaschke product not satisfying the condition \eqref{AC1} of the Ahern-Clark theorem. We have a general result but stated in very different terms, and using very different techniques, than the paper here. In particular, in \cite{HR3}
we obtain two-sided estimates for the reproducing kernels which 
yields more precise results. The results presented here are one-sided estimates but are for general inner functions and not just Blaschke products. 

\section{Proof of the main result}\label{section3}

It is well known that $(I H^2)^{\perp}$ is a reproducing kernel Hilbert space with kernel function 
$$k^{I}_{\lambda}(z) := \frac{1 - \overline{I(\lambda)} I(z)}{1 - \overline{\lambda} z}.$$ It suffices to prove Theorem \ref{main-thm} for $\zeta = 1$. 
If $\|\cdot\|$ denotes the norm in $H^2$, the estimate in \eqref{main-thm} follows from the following result along with the obvious estimate 
$$|f(r)| \leq \|f\| \|k^{I}_{r}\|, \quad f \in (I H^2)^{\perp}, \quad r \in (0, 1).$$

\begin{theorem}\label{mainthm}
Let $I=B S_{\mu}$ be an inner function with 
zeros $\{\lambda_{n}\}_{n \geq 1}$ and associated
singular measure $\mu$ and 
$\vp$ be an admissible function.
If
\bea\label{grcondbis}
  \sum_{n \geq 1}\frac{1-|\lambda_n|}{\vp(|1-\lambda_n|)}+
 \int_{\T}\frac{1}{\vp(|1-e^{it}|)}d\mu(e^{it})<\infty,
\eea
then
\begin{equation} \label{main-k-est}
 \|k^I_{r}\|^2\lesssim \frac{\vp(1-r)}{(1-r)^2}.
\end{equation}
\end{theorem}

\begin{proof}
Our first observation is that since $x\lmto \vp(x)/x$ is bounded, 
\eqref{grcondbis} implies condition \eqref{F1}. By Theorem \ref{thm1.1}
this implies that $\lim_{r \to 1^{-}} |B(r)| = \lim_{r \to 1^{-}} |S_{\mu}(r)| = 1$.
Hence
\[
 \|k_r^I\|^2=\frac{1-|I(r)|^2}{1-r^2}
% \simeq \frac{1-|I(r)|}{1-r}
 =\frac{1-\exp(\log(|I(r)|^2))}{1-r^2}
 =\frac{1-\exp(\log(|B(r)|^2+\log|S_{\mu}(r)|^2))}{1-r^2},
\] 
and since $\log|B(r)|\to 0$ and $\log|S_{\mu}(r)|\to 0$ when
$r\to 1$, we get
\beqa
% \lefteqn
\|k_r^I\|^2&=&
  \frac{1-\exp(\log|B(r)|^2+\log|S_{\mu}(r)|^2)}{1-r^2} \\
 &=&\frac{1-\left(1+\Big(\log|B(r)|^2+\log|S_{\mu}(r)|^2\Big)+
  o\Big(\log|B(r)|^2+\log|S_{\mu}(r)|^2\Big) \right)} {1-r^2}\\
 &\sim&\frac{\log|B(r)|^{-2}+\log|S_{\mu}(r)|^{-2}}{1-r^2}.
\eeqa
Thus to prove the estimate in \eqref{main-k-est} we need to prove 
\begin{equation} \label{main-k-B-est}
 \frac{\log|B(r)|^{-2}}{1-r^2} \lesssim  \frac{\vp(1-r)}{(1-r)^2}
 \end{equation} 
and 
\begin{equation} \label{main-k-S-est}
 \frac{\log|S_{\mu}(r)|^{-2}}{1-r^2} \lesssim  \frac{\vp(1-r)}{(1-r)^2}.
 \end{equation} 

\underline{Case 1: the Blaschke product $B$.}

First note that from the Frostman condition \eqref{F1} we get 
\bea\label{cv-zero}
 \frac{1-|\lambda_n|}{|1-\lambda_n|}\lra 0.
\eea
In particular, from a certain index  $n_0$ on the points $\lambda_n$, $n \geq n_0$,  will be
pseudohyperbolically far from the radius $[0,1)$, i.e., there is a $\delta$
such that for every $n\ge n_0$ and $r\in [0,1)$,
\[
 |b_{\lambda_n}(r)|\ge \delta.
\]
This implies
$$\log \frac{1}{|b_{\lambda_n}(r)|^2} \asymp 1 - |b_{\lambda_n}(r)|^2.$$
A well known calculation shows that 
$$1 - |b_{\lambda_n}(r)|^2 = \frac{(1 - r^2)(1 - |\lambda_{n}|^2)}{|1 - r \overline{\lambda_n} |^2}. 
%= \frac{(1 - r^2)(1 - |\lambda_{n}|^2)}{|e^{i \theta_n} - r r_n |^2}.
$$
Thus 
\bea\label{estimlogB}
 \frac{\log|B(r)|^{-2}}{1-r^2}
 &=&\frac{1}{1-r^2}\sum_{n \geq 1}\log\frac{1}{|b_{\lambda_n}(z)|^2}
 \asymp \sum_{n \geq 1}\frac{1-|\lambda_n|^2}{|1-\overline{\lambda_n} r|^2}.
% &\asymp&
% (1-|r|^2)\sum_{n \geq 1}\frac{(1-r_n^2)}{|e^{i\theta_n}-rr_n|^2}.
%\\
% &\simeq& (1-|z|^2)\sum_n\frac{x_n\theta_n^2}{|e^{i\theta_n}-rr_n|^2}
\eea

Now let $\lambda_n=r_ne^{i\theta_n}$. We need the following two easy estimates:
\begin{equation} \label{Pythag-Lemma}
|1 - \rho  e^{i \theta}|^2 \asymp (1 - \rho )^2 + \theta^2, \quad  \rho \approx 1, \theta \approx 0.
\end{equation}
%\begin{lemma}[Pythagorean type theorem]\label{Pythag-Lemma}
%If $\lambda = r e^{i \theta}$, $r \in (0, 1)$, $\rho \in (0, 1]$, then
%$$|1 - \rho \lambda|^2 \asymp (1 - \rho r)^2 + \theta^2, \quad \rho \approx 1, r \approx 1, \theta \approx 0.$$
%\end{lemma}
%(Though this is a standard result, a proof may be found in \cite{HR3}.)
\bea\label{norms-equiv}
%The norms $
 (|z|^2 + |w|^2)^{1/2}\asymp |z| + |w|,\quad z,w\in\C. %$ are equivalent on $\C^2$. 
\eea
%\end{lemma}
%(constants independent of  $z$ and $w$).

In particular, $|1-\lambda_n|^2\asymp (1-r_n)^2+\theta_n^2$. 
We now remember condition \eqref{cv-zero} which implies
 that $1-r_n=1-|\lambda_n|=o(|1-\lambda_n|)
=o((1-r_n)+\theta_n)$ so that necessarily $1-r_n=o(\theta_n)$.
Hence
\[
 |1-\overline{\lambda}_nr|^2
 \asymp (1-r_nr)^2+\theta_n^2
 =(1-r_n+r_n(1-r))^2+\theta_n^2\asymp (1-r)^2+\theta_n^2.
\]
The estimate in \eqref{estimlogB} yields
\bea\label{estimKernelB}
 \frac{\log|B(r)|^{-2}}{1-r^2}
 &\asymp& \sum_{n \geq 1}\frac{1-|\lambda_n|^2}{|1-\overline{\lambda_n} r|^2}
 \asymp \sum_{n \geq 1}\frac{1-r_n}{(1-r)^2+\theta_n^2}
  \asymp \sum_{\{n:1-r<\theta_n\}}\frac{1-r_n}{\theta_n^2}
 + \sum_{\{n:1-r\ge\theta_n\}}\frac{1-r_n}{(1-r)^2}\nn\\
 &=& \sum_{\{n:1-r<\theta_n\}}\frac{1-r_n}{\theta_n^2}
 + \frac{1}{(1-r)^2}\sum_{\{n:1-r\ge\theta_n\}}(1-r_n).
\eea

Let us discuss each summand in \eqref{estimKernelB} individually. For the first,
we use the fact that $\vp$ is admissible and so $\vp(\theta) \asymp \vp(|1 - e^{i \theta}|)$ to get 
\beqa
 \sum_{\{n:1-r<\theta_n\}}\frac{1-r_n}{\theta_n^2}
 &=&\sum_{\{n:1-r<\theta_n\}}\frac{1-r_n}{\sqrt{\vp(\theta_n)}
 \theta_n^2/\sqrt{\vp(\theta_n)}}\\
 &\le& 
 \underbrace{\left(\sum_{\{n:1-r<\theta_n\}}\frac{1-r_n}{{\vp(\theta_n)}}\right)^{1/2}}_{
 \text{bounded by assumption}}
 \left(\sum_{\{n:1-r<\theta_n\}}\frac{1-r_n}{{\theta_n^4/\vp(\theta_n)}}\right)^{1/2}\\
 &\lesssim&
 \left(\sum_{\{n:1-r<\theta_n\}}\frac{1-r_n}{{\vp(\theta_n)
 (\theta_n^2/\vp(\theta_n))^2}}\right)^{1/2}.
\eeqa
By assumption, $x\to\vp(x)/x^2$ is decreasing. Hence we can bound
$\theta_n^2/\vp(\theta_n)$ below in this last sum by $(1-r)^2/\vp(1-r)$. Hence
\beqa
 \sum_{\{n:1-r<\theta_n\}}\frac{1-r_n}{\theta_n^2}
 \lesssim\frac{\vp(1-r)}{(1-r)^2}
 \left(\sum_{\{n:1-r<\theta_n\}}\frac{1-r_n}{{\vp(\theta_n)}}\right)^{1/2}
 \lesssim\frac{\vp(1-r)}{(1-r)^2} .
\eeqa

For the second sum in \eqref{estimKernelB} we have
\beqa
 \sum_{\{n:1-r\ge\theta_n\}}(1-r_n)
 &=&\sum_{\{n:1-r\ge\theta_n\}}(1-r_n)\frac{\sqrt{\vp(\theta_n)}}{\sqrt{\vp(\theta_n)}}\\
 &\le&
 \underbrace{\left( \sum_{\{n:1-r\ge\theta_n\}}\frac{(1-r_n)}{\vp(\theta_n) }\right)^{1/2}}_{
 \text{bounded by assumption}}
 \left( \sum_{\{n:1-r\ge\theta_n\}}(1-r_n)\vp(\theta_n) \right)^{1/2}\\
 &\lesssim&\sqrt{\vp(1-r)}\left( \sum_{\{n:1-r\ge\theta_n\}}(1-r_n)\right)^{1/2},
\eeqa
where we have used the fact that $\vp$ is increasing. Dividing through the square
root of the sum, this last inequality (and then squaring) 
implies
\[
 \sum_{\{n:1-r\ge\theta_n\}}(1-r_n)\lesssim  \vp(1-r).
\]
This verifies \eqref{main-k-B-est}.

\underline{Case 2: the singular inner factor $S_{\mu}$.}

This case is very similar to the first case. Indeed,
\[
 \frac{\log|S_{\mu}(r)|^{-2}}{1-r^2}=2\int_{\T}\frac{1}{|1-re^{i\theta}|^2}d\mu(e^{i\theta})
  \asymp\int_{\T}\frac{1}{(1-r)^2+\theta^2}d\mu(e^{i\theta})
\]
where we have again used \eqref{Pythag-Lemma}.
% (note that it is no problem here that $|\lambda|=1$).
As in the Blaschke situation we split the integral into two parts depending on
which term in the denominator dominates:
\bea\label{estimSmu}
 \frac{\log|S_{\mu}(r)|^{-2}}{1-r^2}
 &\lesssim&\int_{\{\theta:1-r\le\theta\}}\frac{1}{(1-r)^2+\theta^2}d\mu(e^{i\theta})+
 \int_{\{\theta:1-r\ge \theta\}}\frac{1}{(1-r)^2+\theta^2}d\mu(e^{i\theta})\nn\\
 &\asymp& \int_{\{\theta:1-r\le\theta\}}\frac{1}{\theta^2}d\mu(e^{i\theta})
 +\frac{1}{(1-r)^2}\int_{\{\theta:1-r\ge \theta\}}d\mu(e^{i\theta}).
\eea
Let us consider the first integral.
\beqa
 \int_{\{\theta:1-r\le\theta\}}\frac{1}{\theta^2}d\mu(e^{i\theta})
 &=&\int_{\{\theta:1-r\le\theta\}}
 \frac{1} {\sqrt{\vp(\theta)}\theta^2/\sqrt{\vp(\theta)}}d\mu(e^{i\theta})\\
 &\le&
 \left(\int_{\{\theta:1-r\le\theta\}}
 \frac{1} {{\vp(\theta)}}d\mu(e^{i\theta})\right)^{1/2}
 \left(\int_{\{\theta:1-r\le\theta\}}
 \frac{1} {{\theta^4/{\vp(\theta)}}}d\mu(e^{i\theta})\right)^{1/2}.
\eeqa
Note that $|1-e^{i\theta}|\asymp \theta$. Then using the hypothesis of
admissibility we have $\vp(\theta) \asymp \vp(|1 - e^{i \theta}|)$ %and the fact that $\vp$ is increasing 
and so
\[
\int_{ }
 \frac{1}{{\vp(\theta)}}d\mu(e^{i\theta})
 \asymp \int_{ }
 \frac{1}{{\vp(|1-e^{i\theta}|)}}d\mu(e^{i\theta})
 \]
which is bounded by assumption.
Hence, by the Cauchy-Schwarz inequality,
\beqa
 \int_{\{\theta:1-r\le\theta\}}\frac{1}{\theta^2}d\mu(e^{i\theta})
 \lesssim\left(\int_{\{\theta:1-r\le\theta\}}
 \frac{1} {{\theta^4/{\vp(\theta)}}}d\mu(e^{i\theta})\right)^{1/2}
 =\left(\int_{\{\theta:1-r\le\theta\}}
 \frac{{\vp^2(\theta)}} {{\vp(\theta)\theta^4}}d\mu(e^{i\theta})\right)^{1/2}.
\eeqa
Now using the fact that $x\lra \vp(x)/x^2$ is decreasing we obtain
$\vp^2(\theta)/\theta^4\le (\vp(1-r))^2/(1-r)^4$. Hence
\[
 \int_{\{\theta:1-r\le\theta\}}\frac{1}{\theta^2}d\mu(e^{i\theta})
 \lesssim\frac{\vp(1-r)}{(1-r)^2}\left(\int_{\{\theta:1-r\le\theta\}}
 \frac{{1}} {{\vp(\theta)}}d\mu(e^{i\theta})\right)^{1/2}
 \lesssim \frac{\vp(1-r)}{(1-r)^2}.
\] 

We turn to the second integral in \eqref{estimSmu} to get
\beqa
 \int_{\{\theta:1-r\ge \theta\}}d\mu(e^{i\theta})
 &=&\int_{\{\theta:1-r\ge \theta\}}
 \frac{\sqrt{\vp(\theta)}}{\sqrt{\vp({\theta})}}d\mu(e^{i\theta})\\
 &\le& \left( \int_{\{\theta:1-r\ge \theta\}}
   \vp({\theta}) d\mu(e^{i\theta})\right)^{1/2}
 \left( \int_{\{\theta:1-r\ge \theta\}}
 \frac{1}{{\vp({\theta})}}d\mu(e^{i\theta})\right)^{1/2}.
\eeqa
We have already seen above that
the second factor above is bounded by assumption. 
%For the first, note that
%$|1-e^{i\theta}|\asymp \theta$. Then using the hypothesis (3) of
%admissibility and 
Using the fact that $\vp$ is increasing we get
\beqa
 \int_{\{\theta:1-r\ge \theta\}}d\mu(e^{i\theta})
 \lesssim 
%\left( \int_{\theta:1-r\ge \theta}
% {{\vp(|1-e^{i\theta}|)}}d\mu(e^{i\theta})\right)^{1/2}
  \left( \int_{\{\theta:1-r\ge \theta\}}
 {\vp(\theta)}d\mu(e^{i\theta})\right)^{1/2}
  \le \sqrt{\vp(1-r)}\left( \int_{\{\theta:1-r\ge \theta\}}
 d\mu(e^{i\theta})\right)^{1/2}.
\eeqa
Dividing through by the integral (and then squaring), we obtain
\[
  \int_{\{\theta:1-r\ge \theta\}}d\mu(e^{i\theta})\lesssim \vp(1-r),
\]
which verifies \eqref{main-k-S-est}.
\end{proof}

\section{An example}

The Blaschke situation was
discussed in \cite{HR3} where we obtained two-sided estimates for the
reproducing kernels. It can be shown with concrete examples that the estimates
from Theorem \ref{mainthm} are in general weaker than those obtained in
\cite{HR3} for Blaschke products.
 
%So we will only consider singular inner functions
%here. The 
Let us discuss the simplest case, in fact close enough to a 
Blaschke product, that a singular inner function $S_{\mu}$ with a discrete measure $\mu$. Let
\[
 \mu=\sum_{n \geq 1} \alpha_n\delta_{\zeta_n},
\]
where $\delta_{\zeta_n} \in \T$ and  $\alpha_n$ are positive numbers with $\sum_n\alpha_n<\infty$
guaranteeing that $\mu$ is a finite measure on $\T$. 
Let us fix 
\[
 \zeta_n=e^{i\theta_n}=e^{i/n},\quad 
 \alpha_n=\frac{1}{n^{1+\eps}},\quad n = 1, 2, \ldots.
\]
Also let $\vp(t)=t^{\gamma}$ which defines an admissible function
for $1<\gamma< 2$. In order to have condition \eqref{grcondbis}
it is necessary and sufficient to have
\[
 \sum_n\alpha_n\frac{1}{\vp(|1-e^{i\theta_n}|)}
 \simeq \sum_n\frac{1}{n^{1+\eps}}\frac{1}{\vp(1/n)}
 \simeq \sum_n\frac{n^{\gamma}}{n^{1+\eps}}
 =\sum_n \frac{1}{n^{1+\eps-\gamma}}<\infty
\]
which is equivalent to $ \gamma<\eps$.
We suppose that 
\bea\label{condeps}
 1<\eps<2.
\eea
By Theorem \ref{mainthm} we deduce that
\[
 \|k_r^I\|^2\lesssim \frac{\vp(1-r)}{(1-r)^2}
 =\left(\frac{1}{1-r}\right)^{2-\gamma}.
\] 
In this situation we have 
$$|f(r)| \lesssim \frac{1}{(1 - r)^{1 - \gamma/2}}, \quad f \in (S_{\mu} H^2)^{\perp},$$
which is slower growth than the standard estimate 
$$|f(r)| \lesssim \frac{1}{(1 - r)^{1/2}}, \quad f \in H^2.$$

In this situation, it is actually possible to get a double-sided estimate for the
reproducing kernel: since $\vp$ is admissible, 
Theorem \ref{thm1.1}
implies that $I(r)\lra \eta\in\T$ when $r \to 1^{-}$.
In particular for $r\in (0,1)$, this implies that
\[
  |I(r)|=\exp\left(-\sum_n\alpha_n\frac{1-r^2}{|\zeta_n-r|^2}\right)
 \sim 1-\sum_n\alpha_n\frac{1-r^2}{|\zeta_n-r|^2}.
\]
Let us consider the reproducing kernel of $(S_{\mu} H^2)^{\perp}$ at $r=\rho_N=1-2^{-N}$. Indeed,
\beqa
 \|k^I_{\rho_N}\|^2
 &=&\frac{1-|I(\rho_N)|^2}{1-\rho_N^2}
 \asymp 2^N\left( 1- \exp \left(-\sum_n\alpha_n\frac{1-\rho_N^2}
  {|\zeta_n-\rho_N|^2}\right)\right)\\
 &\asymp& 2^N\left( 1-\left(1-\sum_n\alpha_n\frac{1/2^N}
  {|\zeta_n-\rho_N|^2}\right)\right)\\
 &\asymp& \sum_n\frac{\alpha_n}{|\zeta_n-\rho_N|^2}.
\eeqa
Now using \eqref{Pythag-Lemma}
\[
  |\zeta_n-\rho_N|^2\asymp \frac{1}{n^2}+\frac{1}{2^{2N}},
\]
and so
\beqa
 \|k^I_{\rho_N}\|^2
 &\asymp& \sum_n\frac{\alpha_n}{1/n^2 + 1/2^{2N}}
 = \sum_{n\le 2^N}\frac{\alpha_n}{1/n^2} +\sum_{n>2^N}
  \frac{\alpha_n}{1/2^{2N}}\\
 &\asymp&\sum_{n\le 2^N} \frac{n^2}{n^{1+\eps}}
 +2^{2N}\sum_{n>2^N}\frac{1}{n^{1+\eps}}
  \asymp 2^{(2-\eps)N}\\
 &=&\left(\frac{1}{1-\rho_N}\right)^{2-\eps}
\eeqa
or, equivalently, 
\bea\label{EstimRepKern}
 \|k^I_{\rho_N}\|\asymp \left(\frac{1}{1-\rho_N}\right)^{1-\eps/2}
\eea
(the estimate extends to the whole radius). As a consequence, the estimate
from Theorem \ref{mainthm}
is not optimal, though it is possible to come closer to it by choosing e.g.,
$\vp(t)=t^{\eps}/\log^{1+\gamma}(1/t)$, $\gamma>0$.

\section{A lower estimate}

We finish the paper with a construction of an $f \in (S_{\mu} H^2)^{\perp}$, with $\mu$ the discrete measure discussed in the previous section, 
getting close to the growth given by the norm of the reproducing kernels
thoughout the whole radius $(0,1)$. % (or a Stolz angle at 1).
%Let us have a look to the lower estimate. 
As in \cite{HR3} our
construction will be based on unconditional sequences. 
We need to recall some material on generalized interpolation
in Hardy spaces for which we refer the reader
to \cite[Section C3]{Nik2}.
Let $I=\prod_nI_n$ be a factorization of an inner function $I$ into
inner functions $I_n$, $n\in\N$. 
The sequence $\{I_n\}_{n \geq 1}$ satisfies the generalized Carleson
condition, sometimes called the  Carleson-Vasyunin condition, which we will write $\{I_n\}_{n \geq 1}\in(CV)$,
if there is a $\delta>0$ such that
\bea\label{CV}
 |I(z)|\ge \delta\inf_{n \geq 1}|I_n(z)|,\quad z\in\D.
\eea
In the special case of a Blaschke product $B=B_{\Lambda}$ with simple zeros
$\Lambda=\{\lambda_n\}_{n \geq 1}$ and $I_n=b_{\lambda_n}$, this
is equivalent to the well-known Carleson condition $\inf_n|B_{\Lambda\setminus
\{\lambda_n\}}(\lambda_n)|\ge\delta>0$.

If $\{I_n\}_{n \geq 1}\in (CV)$ then $\{(I_n H^2)^{\perp}\}_{n \geq 1}$ is an unconditional basis
for $(I H^2)^{\perp}$ meaning that every $f\in (I H^2)^{\perp}$ can be 
written uniquely as 
\[
 f=\sum_{n \geq 1} f_n,\quad f_n\in (I_n H^2)^{\perp},
\]
with
\[
 \|f\|^2\asymp\sum_{n \geq 1}\|f_n\|^2.
\] 
In our situation we have $I = S_{\mu}$ and 
\[
 I_n=e^{\dst\alpha_n\frac{\dst z+\zeta_n}{\dst z-\zeta_n}}.
\]
The corresponding 
spaces $(I_n H^2)^{\perp}$ are known to be isometrically isomorphic to the Paley-Wiener
space of analytic functions of exponential type $\alpha_n/2$ and square integrable
on the real axis.
In this situation a sufficient condition for \eqref{CV} is known:
\[
 \sup_{n \geq 1}\sum_{k\neq n}\frac{\mu(\{\zeta_n\})\mu(\{\zeta_k\})}{|\zeta_n-\zeta_k|^2}
 <\infty
\]
(see \cite[Corollary 6, p.\ 247]{Niktr}). So, since $\eps>1$ by
\eqref{condeps}, we have
\[
 \sup_{n \geq 1}\sum_{k\neq n}\frac{1/n^{1+\eps}1/k^{1+\eps}}{ |1/n-1/k|^2}
 =\sup_{n \geq 1}\sum_{k\neq n}\frac{1/n^{\eps-1}1/k^{\eps-1}}{ |n-k|^2} 
 \le \frac{\pi^2}{3}<\infty.
\]
%which is true.
Hence $(I H^2)^{\perp}$ is an $\ell^2$-sum of Paley-Wiener spaces (each of which possesses
for instance the harmonic unconditional basis).
In particular, picking 
\[
 \lambda_n:=r_n\zeta_n=r_n e^{i/n},
 \quad r_n=1-\frac{1}{n},
\]
the sequence $\{K_n\}_{n \geq 1}$, where
\[
 K_n=\frac{k_{\lambda_n}^{I_n}}{\|k_{\lambda_n}^{I_n}\|}\in (I_n H^2)^{\perp},
\]
is an unconditional sequence in $(I H^2)^{\perp}$. Observe that $\Lambda=\{\lambda_n\}_{n \geq 1}$ is
{\it not} a Blaschke sequence. We can introduce the
family of functions
\[
 f_{\beta}:=\sum_{n \geq 1}\beta_nK_n
\]
where $\|f_{\beta}\|^2\asymp\sum_{n \geq 1} |\beta_n|^2<\infty$.
Let us estimate the norms $\|k_{\lambda_n}^{I_n}\|$. First observe
that
\[
%  I_n(\lambda_n)=\exp\left(\alpha_n\frac{\lambda_n+\zeta_m}
% {\lambda_n-\zeta_m}\right)=\exp\left(\alpha_n\frac{r_n+1}
% {r_n-1}\right)=\exp\left(\frac{1}{n^{1+\eps}}\frac{2-n^{1+\eps/2}}
% {-n^{-(1+\eps/2)}}\right)
% =\exp\left(\frac{2-n^{1+\eps/2}}{n^{\eps/2}}\right)
 \alpha_n\frac{\lambda_n+\zeta_n}
 {\lambda_n-\zeta_n}=\alpha_n\frac{r_n+1}
 {r_n-1}=\frac{1}{n^{1+\eps}}\frac{2-1/n}
 {-1/n}
 =-\frac{2-1/n}{n^{\eps}}\lra 0,\quad n\to\infty.
\]
Hence
\beqa
 \|k_{\lambda_n}^{I_n}\|^2
 &=&\frac{1-|I_n(\lambda_n)|^2}{1-r_n^2}
 \asymp \frac{1-|I_n(\lambda_n)|}{1-r_n}
 =\frac{1-\exp\Big(\log|I_n(\lambda_n)|\Big)}
 {1-r_n}\\
 &=&\frac{1-\exp\left(\alpha_n\frac{\lambda_n+\zeta_n}
 {\lambda_n-\zeta_n}\right)}{1-r_n}
 \sim\frac{1-\left(1+\alpha_n\frac{r_n+1}{r_n-1}\right)}{1-r_n}\\
 &\sim&\frac{2\alpha_n}{(1-r_n)^2},
\eeqa
so that
\[
 \|k_{\lambda_n}^{I_n}\|\asymp \sqrt{\frac{\alpha_n}{(1-r_n)^2}}
 =\frac{\sqrt{n^{-(1+\eps)}}}{1/n}=n^{1-1/2-\eps/2}=n^{(1-\eps)/2}.
\] 
Observe now that the $\lambda_n$'s belong to a Stolz domain with vertex at 1.  Indeed,
$$1-|\lambda_n|=1-r_n=1/n\simeq |1-\zeta_n|
\asymp |1-\lambda_n|$$ (this follows from \eqref{Pythag-Lemma}).
For fixed $\beta=\{\beta_n\}_{n \geq 1}$ with $\beta_n\ge 0$ we compute
\[
 \Re f_{\beta}(\lambda_N)
 \simeq\sum_{n \geq 1} {\beta_n}{n^{(\eps-1)/2}}
 \Re\frac{1-\overline{I_n(\lambda_n)}I_n(\lambda_N)}
 {1-\overline{\lambda_n}\lambda_N}.
\]
We have already seen that $\R\ni I_n(\lambda_n)\lra 1$, $n\to\infty$, and
\[
 I_n(\lambda_n)\sim 1-\alpha_n\frac{\dst 1+r_n}{\dst 1-r_n}
 \sim 1-\frac{\dst 2}{\dst n^{\eps}}.
\]
We have to consider
\[
 \alpha_n \frac{\lambda_N+\zeta_n}{\lambda_N-\zeta_n}.
\]
For $n$ or $N$ big enough, $\Re(\lambda_N+\zeta_n)
\asymp \Im(\lambda_N+\zeta_n)\asymp |\lambda_N+\zeta_n|\asymp 1$.
We thus have to consider the denominator. We observe that
by Lemma \ref{Pythag-Lemma}
\bea\label{estimp9}
 |\lambda_N-\zeta_n|=|1-\overline{\zeta_n}\lambda_N|
 \asymp (1-r_N)+\left|\frac{1}{n}-\frac{1}{N}\right|
 =\frac{1}{N}+\left|\frac{1}{n}-\frac{1}{N}\right|
 \asymp\left\{
 \begin{array}{ll}
 \frac{\dst 1}{\dst n}&\text{if }n\le N\\
 \frac{\dst 1}{\dst N}&\text{if }n > N\\ \end{array}
 \right.
\eea
As a consequence,
\[
 \alpha_n \frac{\lambda_N+\zeta_n}{\lambda_N-\zeta_n}\lra 0,
 \quad n\to\infty.
\] 
Again:
\[
 I_n(\lambda_N)\sim 1+\alpha_n\frac{\lambda_N+\zeta_n}{\lambda_N-\zeta_n}.
\]
Hence
\beqa
 1-\overline{I_n(\lambda_n)}I_n(\lambda_N)
 &\sim& 1-\left(1+\alpha_n\frac{\dst r_n+1}{\dst r_n-1}\right)
 \left(1+\alpha_n\frac{\lambda_N+\zeta_n}{\lambda_N-\zeta_n}\right)
 \sim \alpha_n\frac{\dst 1+r_n}{\dst 1-r_n}+
   \alpha_n\frac{\zeta_n+\lambda_N}{\zeta_n-\lambda_N}\\
 &=&\alpha_n\left(\frac{ 1+r_n}{\dst 1-r_n}+
    \frac{\zeta_n+\lambda_N}{\zeta_n-\lambda_N}\right)
 =\alpha_n\frac{ ( 1+r_n)(\zeta_n-\lambda_N) +
     (1-r_n)(\zeta_n+\lambda_N)}{(1-r_n)(\zeta_n-\lambda_N)}\\
 &=&2\alpha_n\frac{\zeta_n-r_n\lambda_N}{(1-r_n)(\zeta_n-\lambda_N)}
 =2\alpha_n\zeta_n\frac{1-\overline{\zeta_n}r_n\lambda_N}{(1-r_n)
   (\zeta_n-\lambda_N)}\\
 &=&2\alpha_n\zeta_n\frac{1-\overline{\lambda_n}\lambda_N}{(1-r_n)
   ({\zeta_n}-\lambda_N)}.
\eeqa
From here we have
\bea\label{simil}
 \frac{1-\overline{I_n(\lambda_n)}I_n(\lambda_N)}
 {1-\overline{\lambda_n}\lambda_N}
 \sim \frac{2\alpha_n\zeta_n}{(1-r_n)({\zeta_n}-\lambda_N)}
 =\frac{2}{n^{\eps}}\frac{\zeta_n}{{\zeta}_n-\lambda_N}.
\eea
We claim that at least for $n\ge 2N$,
\[
 \left|\frac{\zeta_n}{\zeta_n-\lambda_N}\right|
 \asymp \Re\frac{\zeta_n}{\zeta_n-\lambda_N}.
\] 
Indeed,
\[
 \frac{\zeta_n}{\zeta_n-\lambda_N}=\frac{1-\zeta_n
 \overline{\lambda}_N}{|\zeta_n-\lambda_N|^2},
\]
so that for the claim to hold it is sufficient to check that 
\[
 |1-\zeta_n \overline{\lambda}_N|\asymp \Re(1-\zeta_n\overline{\lambda}_N)
\]
for $n\ge 2N$. We have already seen in \eqref{estimp9} that
\[
 |1-\zeta_n \overline{\lambda}_N|\asymp \frac{1}{N},\quad n\ge 2N.
\] 
Now
\[
 \Re(1-\zeta_n \overline{\lambda}_N)=1-r_N\Re \left(e^{i(1/n-1/N)}\right)
 =1-\left(1-\frac{1}{N}\right)\left(\cos\left(\frac{1}{n}-\frac{1}{N}\right)\right)
 \asymp\frac{1}{N},\quad n\ge 2N,
\]
which proves the claim.
We thus can pass in \eqref{simil} to real parts so that for $n\ge 2N$
\beqa
 \Re  \left( \frac{1-\overline{I_n(\lambda_n)}I_n(\lambda_N)}
 {1-\overline{\lambda_n}\lambda_N}\right)
 &\sim& \Re\left(\frac{2}{n^{\eps}}\frac{\zeta_n}{{\zeta}_n-\lambda_N}\right)
 \sim \frac{2}{n^{\eps}}\Re\left( \frac{1-\zeta_n\overline{\lambda}_N}{|\zeta_n-\lambda_N|^2}
 \right)\\
 &\asymp& \frac{2}{n^{\eps}}\frac{1/N}{1/n^2+(1/n-1/N)^2}
 \asymp\frac{2}{n^{\eps}} \frac{ 1/N}{(1/N)^2}\\
 &\asymp& \frac{N}{n^{\eps}},\quad \text{when }n\ge 2N.
\eeqa

Hence
\beqa
 \Re f_{\beta}(\lambda_N)&\gtrsim& 
 \sum_{n \geq 1} \beta_n\frac{1}{n^{(1-\eps)/2}}\frac{\Re(1-\zeta_n\overline{\lambda}_N)}
 {|\zeta_n-\lambda_N|^2}
 \gtrsim N\sum_{n\ge 2N} \frac{\beta_n}{n^{(1+\eps)/2}}.
\eeqa
Pick for instance $\beta_n=n^{-(1+\eta)/2}$, where $\eta>0$ is arbitrary, 
so that obvioulsy $\beta_n\ge 0$ and $\beta\in \ell^2$.
Then
\[
 \Re f_{\beta}(\lambda_N)
 \gtrsim N\sum_{n\ge 2N}\frac{1}{n^{1+(\eps+\eta)/2}}
 \sim N\frac{1}{N^{(\eps+\eta)/2}}=N^{1-\eps/2-\eta/2}
 \asymp\left(\frac{1}{1-|\lambda_N|}\right)^{1-\eps/2-\eta/2}
\]
where $\eta>0$ is arbitrarily small. Compare this with the estimate
of the reproducing kernel \eqref{EstimRepKern}. With better choices of $\beta$ it
is of course clear that we can come closer to the maximal growth given
by the reproducing kernel.

Finally, we point out that when
$I(z)\lmto 1$ when $z\to 1$ in a fixed Stolz domain, 
 it is, in general, particularly 
difficult to decide whether or not a sequence of reproducing kernels for $(I H^2)^{\perp}$, with the parameter in a Stolz domain with vertex at 1, is an
unconditional basis or not. Even when $\sup_n|I(\lambda_n)|<1$,
there is a characterization known for unconditional basis which is, in general, difficult to check.

\bibliography{BBBII}

\end{document}